\documentclass[a4paper,10pt]{article}
\usepackage[utf8]{inputenc}
\usepackage[T1]{fontenc}
\usepackage[english]{babel}
\usepackage{amsmath}
\usepackage{amsthm}
\usepackage{amsfonts}
\usepackage{amssymb}
\usepackage{listings}
\usepackage{graphicx}
\usepackage{mathrsfs}
\usepackage{hyperref}
\usepackage{algpseudocode}
\usepackage{algorithmicx}
\usepackage{algorithm}
\usepackage{subfig}
\usepackage{stmaryrd}
\usepackage{bm}
\usepackage{tikz}
\usepackage{braket}
\usepackage{wasysym}
\usepackage{wrapfig}
\usepackage[framemethod=TikZ]{mdframed}
\usepackage{xcolor}
\usepackage{pstricks}
\usepackage{faktor}

\usepackage[top=2.00cm,bottom=3.00cm,left=2.00cm,right=2.00cm]{geometry}

\theoremstyle{definition}
\newtheorem{definition}{Definition}
\theoremstyle{plain}
\newtheorem{theorem}{Theorem}

\newtheorem{lemma}{Lemma}

 \newcommand{\vertiii}[1]{{\left\vert\kern-0.25ex\left\vert\kern-0.25ex\left\vert #1 
    \right\vert\kern-0.25ex\right\vert\kern-0.25ex\right\vert}}

\title{Ergodic Estimates for Toeplitz Sequences Generated by a Symbol}
\author{Giovanni Barbarino\\\footnotesize Mathematics and Operational Research Unit, Facult\'e polytechnique,\\ \footnotesize  Universit\'e de Mons, Belgium (giovanni.barbarino@umons.ac.be)}
\date{}

\begin{document}
\maketitle

\begin{abstract}
    We analyse the convergence of the ergodic formula for sequences of Toeplitz matrices generated by a symbol. We produce explicit bounds for the convergence rate depending on the size of the matrix, the regularity of the symbol and the regularity of the test function. 
\end{abstract}

\section{Introduction}

Throughout this paper, a matrix-sequence is a sequence of the form $\{A_n\}_n$, where $A_n$ is a square matrix and ${\rm size}(A_n)=n$. Let $C_c(\mathbb R)$ (resp., $C_c(\mathbb C)$) be the space of continuous complex-valued functions with bounded support defined on $\mathbb R$ (resp., $\mathbb C$).
If $A\in\mathbb C^{n\times n}$, the singular values and eigenvalues of $A$ are denoted by $\sigma_1(A),\ldots,\sigma_n(A)$ and $\lambda_1(A),\dots,\lambda_n(A)$, respectively. 
We denote by $\mu_d$ the Lebesgue measure in $\mathbb R^d$. Throughout this paper, ``measurable'' means ``Lebesgue measurable''. Moreover, we denote by $\|g\|_p$ the $L^p$-norm of the function $g$ over its domain (which will be clear from the context). 

We say that a matrix-sequence $\{A_n\}_n$ admits a singular value symbol (or simply a symbol) $f(x)$ when the sampling of $|f(x)|$  over a   uniform   grid in   its domain yields an approximation of the singular values of   $A_n$ that improves   as $n\to \infty$. The  rigorous definition is given below; it is based   on an ergodic formula that must hold for every test   function $F\in C_{c}(\mathbb R)$.   

\begin{definition}[\textbf{asymptotic singular value distribution of a matrix-sequence}]\label{dd}
Let $\{A_n\}_n$ be a matrix-sequence with $A_n$ of size $n$, and let $f:\Omega\subset\mathbb R^d\to\mathbb C$ be measurable with $0<\mu_d(\Omega)<\infty$.
We say that $\{A_n\}_n$ has an asymptotic singular value distribution described by $f$ if the following \textit{ergodic formula} holds:
\begin{equation}\label{eq:ergodic_formula}
\lim_{n\to\infty}\frac1{n}\sum_{i=1}^{d_n}F(\sigma_i(A_n))=\frac1{\mu_d(\Omega)}\int_\Omega\sum_{i=1}^kF(\sigma_i(f(x)))\mathrm{d}x,\qquad\forall\,F\in C_c(\mathbb R).
\end{equation}
In this case, $f$ is called the singular value symbol (or simply the symbol) of $\{A_n\}_n$ and we write $\{A_n\}_n\sim_\sigma f$.
\end{definition}

Among the various matrix-sequences that possess a symbol, Toeplitz sequences generated by $L^1$ functions stand out for their importance.  
A matrix of the form
\begin{equation}\label{mbtm_expr}
\left[f_{i-j}\right]_{i,j=1}^{n}=\begin{bmatrix}
f_0 & f_{-1} & \ f_{-2} & \ \cdots & \ \ \cdots & f_{-(n-1)} \\
f_1 & \ddots & \ \ddots & \ \ddots & \ \ & \vdots\\
f_2 & \ddots & \ \ddots & \ \ddots & \ \ \ddots & \vdots\\
\vdots & \ddots & \ \ddots & \ \ddots & \ \ \ddots & f_{-2}\\
\vdots & & \ \ddots & \ \ddots & \ \ \ddots & f_{-1}\\
f_{n-1} & \cdots & \ \cdots & \ f_2 & \ \ f_1 & f_0
\end{bmatrix},
\end{equation}
whose $(i,j)$-th entry depends only on the difference $i-j$ is called a Toeplitz matrix. In other words, a Toeplitz matrix is a matrix whose entries are constant along each diagonal.
A case of special interest arises when the entries $f_k$ are the Fourier coefficients of a function $f:[-\pi,\pi]\to\mathbb C$ in $L^1([-\pi,\pi])$, i.e.,
\[ f_k=\frac1{2\pi}\int_{-\pi}^\pi f(x){\rm e}^{-{\rm i}kx}{\rm d}x,\qquad k\in\mathbb Z. \]
In this case, the matrix \eqref{mbtm_expr} is denoted by $T_n(f)$ and is referred to as the $n$-th Toeplitz matrix generated by $f$. One can also define the associated infinite-dimensional operator $T(f)$ and represent it as the infinite Toeplitz matrix $\left[f_{i-j}\right]_{i,j\in \mathbb N}$.   
It is not difficult to see that the conjugate transpose of $T_n(f)$ is given by
\begin{equation*}
T_n(f)^H=T_n(\overline f)
\end{equation*}
for every $f\in L^1([-\pi,\pi])$ and every $n$; see, e.g., \cite[Section~6.2]{GLT1}. In particular, $T_n(f)$ is Hermitian whenever $f$ is almost everywhere a real-valued function.\\

The asymptotic singular value distribution of the Toeplitz sequence $\{T_n(f)\}_n$ has been deeply investigated over time, starting from Szeg\"o’s first limit theorem \cite{GS} and the Avram–Parter theorem \cite{Avram,Parter}, up to the works by Tyrtyshnikov–Zamarashkin \cite{Ty96,ZT,ZT'} and Tilli \cite{TilliL1,Tilli-complex}. For more on this subject, see \cite[ Chapter 5]{BoSi}  and \cite[Chapter 6]{GLT1} and references therein. 
Using our notations, the Avram-Parter theorem states that $\{T_n(f)\}_n\sim_\sigma f$ for any  $f\in L^1([-\pi,\pi])$. As a consequence, the ergodic formula \eqref{eq:ergodic_formula} holds with $A_n=T_n(f)$ for any $L^1$ function $f$.

The analysis of the ergodic formula for Toeplitz sequences generated by real-valued symbols $f$ is also linked to the asymptotic study of the eigenvalues and spectral distribution of the matrices $T_n(f)$. Several works enquire about the existence of an asymptotic expansion of the eigenvalues when we impose certain conditions on $f$. 
See \cite{SL} and references therein for the case of (modified) simple loop functions, that are in particular nonnegative $C^1_{per}[-\pi,\pi]$ functions with $f(0)=f'(0)=0$ being a zero of order at most 2 and exactly another extremum $f'(\xi)=0$ in the open interval $(-\pi,\pi)$. 
For the case of (some) smooth even function $f$ that is monotonous on $[0,\pi]$ see \cite{Sven1,Sven2}.

Several have been the studies on the ergodic formula \eqref{eq:ergodic_formula} for Toeplitz matrices, for example about the  classes of test functions satisfying it, (see \cite{Serra02,Serra03}).

In this context, Widom proved a more precise result \cite{Widom}. He showed that, for essentially bounded symbols $f$ with $$\vertiii{f}^2:= \sum_{k=-\infty}^{\infty} \lvert k\rvert \lvert f_k\rvert^2<\infty,$$ we have 
\begin{equation}\label{eq:ergodic_Widom}
    \lim_{n\to \infty} \left\{ 
    \sum_{j=1}^n G(\sigma_j(T_n(f))^2) - \frac n {2\pi} \int_{-\pi}^\pi G(\lvert f(\theta)\rvert^2)d\theta
    \right\} = Tr[G(T(\overline f)T(f)) + G(T(f)T(\overline f)) - 2T(G(\lvert f\rvert^2)) ]
\end{equation}
for any $C^3$ function $G(x)$ defined at least on an interval containing  $\sigma_j(T_n(f))^2$ and $\sigma_j(T(f))^2$  for all $j,n$.  This implies that 
\[\left\lvert 
\frac 1n
\sum_{j=1}^n F(\sigma_j(T_n(f))) - \frac 1 {2\pi} \int_{-\pi}^\pi F(\lvert f(\theta)\rvert)d\theta 
\right\rvert = O\left(\frac 1n\right)
\]
thus providing an estimate for the convergence rate of the ergodic formula \eqref{eq:ergodic_formula} for any function $F(x)\in C_c(\mathbb R)$ that is an extension a function of the form $G(x^2)$ with $G(x)$ as specified above. Notice that the hypotheses on $f$ are satisfied whenever $f\in C^1_{per}[-\pi,\pi]$, since in this case $\vertiii f^2\le \|f'\|_2^2$ and $f$ is bounded.

For concrete applications, as in \cite{Bosonic,Bosonic2}, we need an explicit bound of the form
\begin{equation} \label{eq:ergodic_bound}
\left\lvert 
\frac 1n
\sum_{j=1}^n F(\sigma_j(T_n(f))) - \frac 1 {2\pi} \int_{-\pi}^\pi F(\lvert f(\theta)\rvert)d\theta 
\right\rvert \le c(n)
\end{equation}
holding for any $n\ge n_0$ with an explicit $n_0$. Equation \eqref{eq:ergodic_Widom} shows that, for some choices of $F(x)$ and $f(x)$,  we can take $c(n) = O(1/n)$, but the dependence of $c(n)$ on $f$ and $F$ is not specified. 
In the first main result of this paper (Theorem~\ref{th:1}), for the same choices of $F(x)$ and $f(x)$, 
we provide an explicit bound $c(n)$ in \eqref{eq:ergodic_bound}. 
\begin{theorem}\label{th:1}
   Let $f\in L^\infty([-\pi,\pi])$ with $\vertiii{f}^2<\infty$ and let $G\in C^3_c(\mathbb R)$. Then, for every $n\ge 1$, \[      \left\lvert \sum_{j=1}^n  G(\sigma_j(T_n(f))^2) - \frac n {2\pi} \int_{-\pi}^\pi G(\lvert f(\theta)\rvert^2)d\theta \right\rvert\le 
      2\vertiii f^2 \left[ c_1 + 2c_2\|f\|_\infty^2 \right]
    \]
where $c_1= 2\|G'\|_1+\sqrt 2\|G''\|_2$ and $c_2 =  2\|G''\|_1+\sqrt 2\|G'''\|_2$.
\end{theorem}

Notice that the test functions are limited to being highly regular, whereas in several cases one would like to also study the convergence of the ergodic formula \eqref{eq:ergodic_formula} in case the test function $F(x)$ is just Lipschitz and not differentiable (e.g. when $F(x)$ is a truncation of the ReLU function).  We explore this case in a separate paper \cite{Bosonic} where we find a convergence slower than the $O\left(\frac 1n\right)$ of \eqref{eq:ergodic_Widom} depending also on the regularity of the symbol $f(\theta)$.

\section{Known results}

\subsection{Function norms}
Let $A$ be a continuous operator on $\ell_2(\mathbb N)$. Since the singular values are well defined, one can also compute the $p$-Schatten norms $\|A\|_p$ with $\infty \ge p\ge 1$, where $\|A\|_1$ is the Trace norm, $\|A\|_2$ is the Hilbert-Schmidt norm and $\|A\|_\infty = \|A\|$ is the operator norm. Notice that they may also take value $+\infty$ on non-bounded operators. Moreover, when an Hilbert-Schmidt operator $A$ is expressed as a (possibly infinite) matrix, then its norm $\|A\|_2^2$ is equal to the sum of the squared magnitude of its entries. 

An important result that links together the different norms is given by the H\"older theorem.
\begin{theorem}[H\"older]
\label{th:Holder}If $\frac 1r = \frac 1p+\frac 1q$ with $\infty\ge r,p,q\ge 1$, then
\[
\|AB\|_r\le \|A\|_p\|B\|_q
\]
where by convention, $1/\infty =0$. 
In particular, for any $\infty \ge p\ge 1$
\begin{equation*}\label{eq:trace_HS_norm}
    \|AB\|_p \le \|A\|\|B\|_p, \qquad 
\lvert Tr(AB)\rvert \le \|AB\|_1\le \|A\|_2\|B\|_2.
\end{equation*}
\end{theorem}

Let us define the following norm, that will be central in the forthcoming discussions.
\begin{definition}
\begin{equation}
    \label{eq:norm_vertiii} \vertiii{f}^2:= \sum_{k=-\infty}^{\infty} \lvert k\rvert \lvert f_k\rvert^2<\infty. 
\end{equation} \end{definition}

If now $G\in Lip(L)$, then it can be proved that the composition $G\circ f$ has still bounded norm $\vertiii\cdot$.
\begin{lemma}
\cite{Widom2} If $G\in Lip(L)$ with domain at least on the essential range of $f$, and $\vertiii f<\infty$,  then
\[
\vertiii{G\circ f}\le L \vertiii f.
\]
\end{lemma}
In particular, if the range of $f$ is inside the complex ball $D(0,M)$,
\begin{equation}\label{eq:Lipschitz_norm}
    \vertiii{\lvert f\rvert^2} \le 2M \vertiii f,\qquad
\vertiii{e^{\textnormal is\lvert f\rvert^2}} \le 2M|s| \vertiii f.
\end{equation}

\subsection{Estimates for the Fourier transform and series}

For the main results, we need some estimates on the moments of the Fourier transform of $C^3$ functions, and on the truncated Fourier series of $C^k_{per}[-\pi,\pi]$ functions. 

\begin{lemma}
   \label{lem:FT_moments} For any $G\in C_c^3(\mathbb R)$, let $\hat G$ be its Fourier transform. Then $\hat G\in L_1(\mathbb R)$ and $\hat G(x)x^2\in L_1(\mathbb R)$. Moreover,
\[
    \|\hat G(x)x\|_1 \le 
    2\|G'\|_1 + \sqrt 2\|G''\|_2,\qquad 
    \|\hat G(x)x^2\|_1 \le 
    2\|G''\|_1 + \sqrt 2\|G'''\|_2.
\]
\end{lemma}
\begin{proof}
Given any $H\in C_c^1(\mathbb R)$ one has that $|\hat{H'}(x)|=|x\hat H(x)|$, so by H\"older Theorem \ref{th:Holder} and the Plancherel identity,
\begin{align*}
\|\hat H(x)\|_1 &= \int_{|x|\le 1} |\hat H(x)|dx 
    +  \int_{|x|> 1} |\hat H(x)|dx  \\
& \le 2\|H\|_1 + \|x\hat H(x)\|_2\sqrt{\int_{|x|>1} \frac{1}{x^2} dx}
\\&= 2\|H\|_1 + \sqrt 2\|\hat{ H'}(x)\|_2 = 2\|H\|_1 + \sqrt 2\| H'(x)\|_2
\end{align*}
Since $G\in C_c^3(\mathbb R)$, then surely $\hat G(x)\in L_1(\mathbb R)$, $|\hat G(x)x^2|=|\hat{G''}(x)|\in L_1(\mathbb R)$ and
\begin{alignat*}{2}
    \|\hat G(x)x\|_1 &= \|\hat {G'}\|_1 &&\le 
    2\|G'\|_1 + \sqrt 2\| G''\|_2,\\ 
    \|\hat G(x)x^2\|_1 &= \|\hat {G''}\|_1 &&\le 
    2\|G''\|_1 + \sqrt 2\|G'''\|_2.
\end{alignat*}
\end{proof}

\subsection{Toeplitz and Hankel matrices }

Here we report the original result by Widom 

\begin{theorem}\cite[Theorem 1]{Widom}\label{Th:Widom}
Suppose $f\in L^\infty([-\pi,\pi])$ and $\vertiii{f}<\infty$. Let $M$ be the supess of $\lvert f\rvert$, and let $m$ be the distance between the convex hull of the essential range of $f$ and the origin of the complex plane. Suppose $G\in C^3[m^2,M^2]$. Then
\begin{equation}\label{eq:ergodic}
    \lim_{n\to \infty} \left\{ 
    \sum_{j=1}^n G(\sigma_j(T_n(f))^2) - \frac n {2\pi} \int_{-\pi}^\pi G(\lvert f(\theta)\rvert^2)d\theta
    \right\} = Tr[G(T(\overline f)T(f)) + G(T(f)T(\overline f)) - 2T(G(\lvert f\rvert^2)) ].
\end{equation}
\end{theorem} 
The relation \eqref{eq:ergodic} is well-defined due to an additional lemma.

\begin{lemma}\label{lem:bounds_on_sv} \cite[Lemma 1.2]{Widom}
The singular value of $T_n(f)$ and $T(f)$ belong to $[m,M]$, where $M$ is the supess of $\lvert f\rvert$, and $m$ is the distance between the convex hull of the essential range of $f$ and the origin of the complex plane. 
\end{lemma}

Theorem \ref{Th:Widom} has been proved by the repeated use of properties of finite and infinite matrices such as Toeplitz and Hankel matrices. 
In fact, the norm $\vertiii{\cdot}$ defined in \eqref{eq:norm_vertiii} is deeply linked with the infinite Hankel matrices associated to $L^1$ functions: given $g\in L^1[-\pi,\pi]$, the associated Hankel matrix is $H(g) := [g_{i+j-1}]_{i,j\ge 1}$ and it is easy to see that $H(g)$, $H(\overline g)$ and $H(\tilde g)$ are all Hilbert-Schmidt operators whenever $\vertiii g$ is not infinite, since
\begin{align}\label{eq:All_Hankel_are_HS}
 \nonumber  \|H(g)\|_2^2 &=  \sum_{k=1}^{\infty} \lvert k\rvert \lvert g_k\rvert^2 \le \vertiii g^2,\\
\|H(\overline g)\|_2^2 =  \sum_{k=1}^{\infty} \lvert k\rvert \lvert \overline{g_{-k}}\rvert^2 \le \vertiii g^2,&\qquad
\|H(\tilde g)\|_2^2 =  \sum_{k=1}^{\infty} \lvert k\rvert \lvert g_{-k}\rvert^2 \le \vertiii g^2,
\end{align}
where $\overline g(x)$ is the complex conjugate of $g(x)$, whereas $\tilde g(x) = g(-x)$. These matrices are present also in classical results about the product of Toeplitz matrices.

\begin{lemma}\label{lem:product_of_toeplitz} \cite[Lemma 1.1]{Widom}
Given $f,g\in L^\infty([-\pi,\pi])$, let $\tilde f(\theta) = f(-\theta)$. Then 
\begin{align*}
    T(fg)-T(f)T(g) &= H(f) H(\tilde g),\\
    T_n(fg)-T_n(f)T_n(g) &= P_nH(f) H(\tilde g)P_n + Q_nH(\tilde f) H(g) Q_n.
\end{align*}
and $\|H(f)\|_2\le \vertiii f$.  
\end{lemma}
Here $P_n$ is the projection on the first $n$ coordinates of $\ell_2(\mathbb Z_+)$, i.e. the infinite matrix $(P_n)_{i,j} = 1$ for $i=j\le n$ and zero otherwise. Instead $Q_n$ projects on the first $n$ coordinates and then it inverts their order, i.e. the infinite matrix $(Q_n)_{i,j}=1$ for $i+j=n+1$ and zero otherwise.\\

For dealing with (infinite-dimensional) operators, we will need of some results about the exponentiation and integration of operators, that can be proved by expanding the exponential functions into their Taylor series.
\begin{lemma}
 \label{lem:exponential_of_operators}   For any continuous operator $A,B$ from $\ell_2(\mathbb Z_+)$ to itself and any $t\in \mathbb R$ we have
\begin{equation}\label{eq:difference_of_exponential}
    e^{\textnormal itA} - e^{\textnormal itB} = \textnormal i \int_0^t e^{\textnormal isA}(A-B) e^{\textnormal i(t-s)B} ds. 
\end{equation}
Moreover, for any real function $g\in L^\infty([-\pi,\pi])$ and any $n\in \mathbb N$ and $t\in \mathbb R$,  
\begin{equation}\label{eq:other_diff}
   T_n(e^{\textnormal itg}) - e^{\textnormal itT_n(g)} = \textnormal i \int_0^t  
   [ T_n(e^{\textnormal isg}g) - T_n(e^{\textnormal isg}) T_n(g) ]
   e^{\textnormal i(t-s)T_n(g)} ds. 
\end{equation}
and the same holds with $T$ instead of $T_n$. 
\end{lemma}

Eventually, we report a classic result linking Toeplitz matrices to their symbols through their $p$-Schatten norm and $p$-norm. 
\begin{theorem}{\cite[Theorem 6.2]{GLT1}}\label{th:norm_toeplitz_symbol} 
    Let $f\in L^p([-\pi,\pi])$ for some $p\in [1,\infty]$. Then
    $$\|T_n(f)\|_p\le \left(\frac n{2\pi}\right)^{1/p} \|f\|_p$$
    where for any matrix $A$,  $\|A\|_p$ is its $p$-Schatten norm. 
\end{theorem}

\section{Widom Estimation}
This section is dedicated to the proof of Theorem  \ref{th:1}.
{
\renewcommand{\thetheorem}{\ref{th:1}}
\begin{theorem}
Suppose $f\in L^\infty([-\pi,\pi])$ with $\vertiii{f}^2<\infty$ and $G\in C^3_c(\mathbb R)$. Then \[      \left\lvert \sum_{j=1}^n  G(\sigma_j(T_n(f))^2) - \frac n {2\pi} \int_{-\pi}^\pi G(\lvert f(\theta)\rvert^2)d\theta \right\rvert\le 
      2\vertiii f^2 \left[ c_1 + 2c_2\|f\|_\infty^2 \right]
    \]
where $c_1= 2\|G'\|_1+\sqrt 2\|G''\|_2$ and $c_2 =  2\|G''\|_1+\sqrt 2\|G'''\|_2$.
\end{theorem}
\addtocounter{theorem}{-1}
} 

\begin{proof}
First of all, notice that $\frac 1 {2\pi} \int_{-\pi}^\pi G(\lvert f(\theta)\rvert^2)d\theta$ is the $0$-th Fourier coefficient $g_0$ of the bounded function $g(\theta):= G(|f(\theta)|^2)$, that is surely in $L^2[-\pi,\pi]$, and $g_0$ is the $(1,1)$ element of $T(g)$ and $T_n(g)$ for any $n$. Moreover, $\sigma_j(T_n(f))^2$ correspond to the eigenvalues of the matrix $T_n(f)^HT_n( f)=T_n(\overline f)T_n( f)$, and by definition, for any diagonalizable matrix $A$ the matrix $G(A)$ is also diagonalizable with eigenvalues $G(\lambda_j(A))$. As a consequence,
\begin{align}
 \nonumber\label{eq:th1_1}   \sum_{j=1}^n G(\sigma_j(T_n(f))^2) - \frac n {2\pi} \int_{-\pi}^\pi G(\lvert f(\theta)\rvert^2)d\theta
    &=  \sum_{j=1}^n G(\lambda_j(T_n(\overline f)T_n( f))) -
    n T(G(\lvert f(\theta)\rvert^2))_{1,1} \\
    &= Tr\left [    
    G(T_n(\overline f)T_n( f) ) 
    - 
    T_n(G(\lvert f(\theta)\rvert^2))
    \right ]
\end{align}
Since $G\in L^2(\mathbb R)$, we can write it through inverse Fourier Transform as $G(y)=\int \hat G(t)e^{\textnormal ity}dt$. Moreover, by Lemma \ref{lem:FT_moments} we know that
\[
\left|\hat G(t) e^{\textnormal it|f(\theta)|^2} e^{-\textnormal ik\theta} \right| = \left|\hat G(t)\right|\in L^1(\mathbb R)
\]
so by Fubini-Tonelli Theorem we have
\begin{align*}
 g_k = [G(|f(\theta)|^2)]_k &= \frac 1{2\pi}\int_{-\pi}^\pi G(|f(\theta)|^2)e^{-\textnormal i k\theta}d\theta   \\
 &=\frac 1{2\pi}\int_{-\pi}^\pi \int \hat G(t) e^{\textnormal it|f(\theta)|^2}  e^{-\textnormal i k\theta}dt\,d\theta \\
 &= \int \frac 1{2\pi}\int_{-\pi}^\pi \hat G(t) e^{\textnormal it|f(\theta)|^2}  e^{-\textnormal i k\theta}d\theta \,dt
 =\int     \left[\hat G(t) e^{\textnormal it|f(\theta)|^2} \right]_k dt
\end{align*}
and as a consequence, $T_n(G(\lvert f(\theta)\rvert^2)=\int T_n(
    \hat G(t) e^{\textnormal it\lvert f(\theta)\rvert^2 })dt$. We can thus rewrite \eqref{eq:th1_1}
 as 
 \begin{align}
\label{eq:th1_4}\nonumber    Tr\left [    
    G(T_n(\overline f)T_n( f) ) 
    - 
    T_n(G(\lvert f(\theta)\rvert^2)
    \right ]
    &=
  Tr\left [  
  \int \hat G(t) e^{\textnormal itT_n(\overline f)T_n( f) }
    - 
    T_n(
    \hat G(t) e^{\textnormal it\lvert f(\theta)\rvert^2 })dt
    \right ]  \\
 \nonumber   &= 
     Tr\left [  
  \int \hat G(t) 
  [e^{\textnormal itT_n(\overline f)T_n( f) }
    - 
    T_n(
  e^{\textnormal it\lvert f(\theta)\rvert^2 })
  ]dt
    \right ] \\
     &= 
      \int \hat G(t) 
  Tr\left [  e^{\textnormal itT_n(\overline f)T_n( f) }
    - 
    T_n(
  e^{\textnormal it\lvert f(\theta)\rvert^2 })
  \right ]dt
    \end{align}
Let us forget the test function $G$ for now and focus on the matrix inside the square brackets. Split it into two pieces
\begin{equation}\label{eq:th1_2}
     e^{\textnormal itT_n(\overline f)T_n( f) } - T_n(e^{\textnormal it\lvert f\rvert^2}) = 
\left( e^{\textnormal itT_n(\overline f)T_n( f) }  
-     e^{\textnormal itT_n(\lvert f\rvert^2)}  \right)
+ \left(    e^{\textnormal itT_n(\lvert f\rvert^2)}
-     T_n(e^{\textnormal it\lvert f\rvert^2})  \right).
\end{equation}
By \eqref{eq:difference_of_exponential} and Lemma \ref{lem:product_of_toeplitz}, we can rewrite the first piece as
\begin{align}\label{eq:th1_3}
\nonumber e^{\textnormal itT_n(\overline f)T_n( f) }  
-     e^{\textnormal itT_n(\lvert f\rvert^2)} &= \textnormal i \int_0^t e^{\textnormal isT_n(\overline f)T_n( f) }[T_n(\overline f)T_n( f) -T_n(\lvert f\rvert^2)] e^{\textnormal i(t-s)T_n(\lvert f\rvert^2)} ds\\
&=-\textnormal i \int_0^t e^{\textnormal isT_n(\overline f)T_n( f) }
[P_nH(\overline f) H(\tilde f)P_n + Q_nH(\tilde{\overline f}) H(f) Q_n] 
e^{\textnormal i(t-s)T_n(\lvert f\rvert^2)} ds.
\end{align}
Notice that $\|P_n\|=\|Q_n\|=1$ and that for any normal matrix $A$, we have $\|e^{\textnormal iA}\|=1$. As a consequence, we can evaluate the trace of \eqref{eq:th1_3} using H\"older Theorem \ref{th:Holder} and the relations in \eqref{eq:All_Hankel_are_HS} as
\begin{align}\label{eq:th1_5}
\nonumber   \Big\lvert Tr\Big [  &
    e^{\textnormal itT_n(\overline f)T_n( f) }  
-     e^{\textnormal itT_n(\lvert f\rvert^2)}  \Big ]\Big\rvert\\
\nonumber  \le &
 \int_0^{\lvert t\rvert}
 \left\lvert Tr\left [
 e^{\textnormal isT_n(\overline f)T_n( f) }
[P_nH(\overline f) H(\tilde f)P_n + Q_nH(\tilde{\overline f}) H(f) Q_n] 
e^{\textnormal i(t-s)T_n(\lvert f\rvert^2)}
\right ]\right\rvert ds
  \\
\nonumber  \le &
 \int_0^{\lvert t\rvert}
 \left\lvert Tr\left [
 e^{\textnormal isT_n(\overline f)T_n( f) }
P_nH(\overline f) H(\tilde f)P_n
e^{\textnormal i(t-s)T_n(\lvert f\rvert^2)}
\right ]\right\rvert ds\\
\nonumber& +
\int_0^{\lvert t\rvert}
 \left\lvert Tr\left [
 e^{\textnormal isT_n(\overline f)T_n( f) }
 Q_nH(\tilde{\overline f}) H(f) Q_n 
e^{\textnormal i(t-s)T_n(\lvert f\rvert^2)}
\right ]\right\rvert ds\\
\nonumber   \le &
 \int_0^{\lvert t\rvert}
 \|
 e^{\textnormal isT_n(\overline f)T_n( f) }
P_nH(\overline f)\|_2 \| H(\tilde f)P_n
e^{\textnormal i(t-s)T_n(\lvert f\rvert^2)}
\|_2 ds\\
\nonumber& +
\int_0^{\lvert t\rvert}
 \|
 e^{\textnormal isT_n(\overline f)T_n( f) }
 Q_nH(\tilde{\overline f})\|_2 
 \| H(f) Q_n e^{\textnormal i(t-s)T_n(\lvert f\rvert^2)}\|_2 ds  \\
  \le &
 \int_0^{\lvert t\rvert}
 \|H(\overline f)\|_2 \| H(\tilde f)\|_2 ds
 +
\int_0^{\lvert t\rvert}
 \|H(\tilde{\overline f})\|_2  \| H(f) \|_2 ds \le  2{\lvert t\rvert}  \vertiii f^2 
\end{align}
By \eqref{eq:other_diff} and Lemma \ref{lem:product_of_toeplitz}, we can rewrite the second piece in \eqref{eq:th1_2} as
\begin{align*}
\nonumber    e^{\textnormal itT_n(\lvert f\rvert^2)}
-     T_n(e^{\textnormal it\lvert f\rvert^2})
    &= - 
    \textnormal i \int_0^t  
   [ T_n(e^{\textnormal is\lvert f\rvert^2}\lvert f\rvert^2) - T_n(e^{\textnormal is\lvert f\rvert^2}) T_n(\lvert f\rvert^2) ]
   e^{\textnormal i(t-s)T_n(\lvert f\rvert^2)} ds\\
   &= 
  -  \textnormal i \int_0^t  
   [    P_nH(e^{\textnormal is\lvert f\rvert^2}) H(\lvert \tilde f\rvert^2)P_n + Q_nH(e^{\textnormal is\lvert \tilde f\rvert^2}) H(\lvert f\rvert^2) Q_n ]
   e^{\textnormal i(t-s)T_n(\lvert f\rvert^2)} ds,
\end{align*}
and using H\"older Theorem \ref{th:Holder} and the relations in \eqref{eq:All_Hankel_are_HS}, \eqref{eq:Lipschitz_norm}, we get
\begin{align}\label{eq:th1_7}
\nonumber
  \Big\lvert Tr\Big [  &  e^{\textnormal itT_n(\lvert f\rvert^2)}
    - 
    T_n(
  e^{\textnormal it\lvert f(\theta)\rvert^2 }) 
  \Big ]\Big\rvert\\
\nonumber  \le &
 \int_0^{\lvert t\rvert}  
\left\lvert Tr\left [   [    P_nH(e^{\textnormal is\lvert f\rvert^2}) H(\lvert \tilde f\rvert^2)P_n + Q_nH(e^{\textnormal is\lvert \tilde f\rvert^2}) H(\lvert f\rvert^2) Q_n ]
   e^{\textnormal i(t-s)T_n(\lvert f\rvert^2)} 
   \right ]\right\rvert ds  
   \\
\nonumber  \le &
  \int_0^{\lvert t\rvert}  
\left\lvert Tr\left [  P_nH(e^{\textnormal is\lvert f\rvert^2}) H(\lvert \tilde f\rvert^2)P_n 
   e^{\textnormal i(t-s)T_n(\lvert f\rvert^2)} 
   \right ]\right\rvert ds
    \\\nonumber&   + \int_0^{\lvert t\rvert}  
\left\lvert Tr\left [   Q_nH(e^{\textnormal is\lvert \tilde f\rvert^2}) H(\lvert f\rvert^2) Q_n 
   e^{\textnormal i(t-s)T_n(\lvert f\rvert^2)} 
   \right ]\right\rvert ds \\
\nonumber  \le &
  \int_0^{\lvert t\rvert}  
\|   P_nH(e^{\textnormal is\lvert f\rvert^2}) \|_2 
\| H(\lvert \tilde f\rvert^2)P_n   e^{\textnormal i(t-s)T_n(\lvert f\rvert^2)}\|_2 
  ds    \\
\nonumber  &   + \int_0^{\lvert t\rvert}  
\|  Q_nH(e^{\textnormal is\lvert \tilde f\rvert^2})\|_2 
\| H(\lvert f\rvert^2) Q_n    e^{\textnormal i(t-s)T_n(\lvert f\rvert^2)}\|_2 ds 
\\
\nonumber  \le &
  \int_0^{\lvert t\rvert}  
\|  H(e^{\textnormal is\lvert f\rvert^2}) \|_2 
\| H(\lvert \tilde f\rvert^2)\|_2 
  ds     + \int_0^{\lvert t\rvert}  
\|  H(e^{\textnormal is\lvert \tilde f\rvert^2})\|_2 
\| H(\lvert f\rvert^2) \|_2 ds 
\\
  \le &\ 
 2\int_0^{\lvert t\rvert}  
\vertiii {e^{\textnormal is\lvert f\rvert^2}}
\vertiii {\lvert f\rvert^2} 
  ds     
 \le 
8\|f\|_\infty^2\vertiii {f}^2\int_0^{\lvert t\rvert}  
|s|    ds  
= 4   t^2 \|f\|_\infty^2\vertiii f^2.
\end{align}
Substituting \eqref{eq:th1_5} and \eqref{eq:th1_7} inside first \eqref{eq:th1_2}, then \eqref{eq:th1_4} and \eqref{eq:th1_1} , we find
\begin{align*}
    \left\lvert \sum_{j=1}^n \right.&\left. G(\sigma_j(T_n(f))^2) - \frac n {2\pi} \int_{-\pi}^\pi G(\lvert f(\theta)\rvert^2)d\theta \right\rvert \\
    &= 
    \left\lvert   \int \hat G(t) 
  Tr\left [  e^{\textnormal itT_n(\overline f)T_n( f) }
    - 
    T_n(
  e^{\textnormal it\lvert f(\theta)\rvert^2 })
  \right ]dt
  \right\rvert\\
   &\le 
      \int\lvert \hat G(t) \rvert 
  \left\lvert Tr\left [  e^{\textnormal itT_n(\overline f)T_n( f) }
    - 
    T_n(
  e^{\textnormal it\lvert f(\theta)\rvert^2 }) 
  \right ]\right\rvert dt\\
    &\le 
      2\vertiii f^2 \int\lvert \hat G(t) \rvert 
  ( {\lvert t\rvert} + 2 t^2\|f\|_\infty^2) dt
\end{align*}
that is a finite quantity as long as $\hat G(t)t^2,\hat G(t)t\in L^1$, that is assured by the hypothesis that $G\in C^3_c$ and Lemma \ref{lem:FT_moments}. In particular, we obtain the bound
\[
 \left\lvert \sum_{j=1}^n  G(\sigma_j(T_n(f))^2) - \frac n {2\pi} \int_{-\pi}^\pi G(\lvert f(\theta)\rvert^2)d\theta \right\rvert
 \le 
      2\vertiii f^2 
     \left( 
    c_1
     +
       2 \|f\|_\infty^2 
      c_2
       \right),
\]
\[
c_1 =  2\|G'\|_1 + \sqrt 2\|G''\|_2, \qquad c_2 = 2\|G''\|_1 + \sqrt 2\|G'''\|_2.
\]
\end{proof}

\section{Conclusions and Future Works}

Up to our knowledge, this article presents the first ever bounds of the rate of convergence for the ergodic formula \eqref{eq:ergodic_formula} for sequences of Toeplitz matrices generated by a symbol that works for any size $n$. In fact, previous works always focus on the asymptotic expansion of the ergodic formula or on the spectral distribution of the matrices $T_n(f)$, and in the latest case, always with regular functions $f$ or with additional hypotheses. 

Much is left to be explored. Almost all the computed bounds are not strict and can be improved or adapted to other characteristics of the symbol $f$ and the test function $F$.
For example, a missing analysis is the case in which the test function $F(x)$ is an indicator of a real interval (it does not follow from Theorem \ref{th:1} since it is not a continuous function) that would also be more tightly linked to the analysis of the spectral distribution of $T_n(f)$.

\section{Acknowledgement}
The author thanks Professor Stefano Serra-Capizzano and Professor Carlo Garoni for their support and help in the writing of the manuscript.
The author acknowledges the support by the European Union (ERC consolidator grant, eLinoR, no 101085607).
The author is a member of the Research Group GNCS (Gruppo Nazionale per il Calcolo Scientifico) of INdAM
(Istituto Nazionale di Alta Matematica).

\end{document}